\documentclass[12pt]{amsart}
\usepackage[margin=1in]{geometry}
\usepackage{amsmath}
\usepackage{amssymb}
\usepackage{amsthm}
\usepackage{enumerate}
\usepackage{yfonts}
\makeatletter
\def\imod#1{\allowbreak\mkern10mu({\operator@font mod}\,\,#1)}
\makeatother
\usepackage{graphicx}

\newtheorem{lemma}{Lemma}[section]

\theoremstyle{definition}

\begin{document}

\title{Revisiting The Riemann Zeta Function at Positive Even Integers}
\author{Krishnaswami Alladi}
\author{Colin Defant}

\begin{abstract}
Using Parseval's identity for the Fourier coefficients of $x^k$, we provide a new proof that $\zeta(2k)=\dfrac{(-1)^{k+1}B_{2k}(2\pi)^{2k}}{2(2k)!}$.  
\end{abstract}

\maketitle 

\section{Introduction} 
One of the most famous mathematical problems was the evaluation of
\[
\sum^{\infty}_{n=1}\frac{1}{n^2}.
\]
The Italian mathematician Pietro Mengoli originally posed this problem in 1644. The problem was later popularized by the Bernoullis who lived in Basel, Switzerland, so this
became known as the \emph{Basel Problem} \cite{Havil}. Leonhard Euler solved this brilliantly
when he was just twenty-eight years old by showing that the sum in question was equal to
$\pi^2/6$. Indeed, this was Euler's first mathematical work, and it brought
him world fame. In solving the Basel Problem, Euler found closed-form
evaluations more generally of
\[
\zeta(2k)=\sum^{\infty}_{n=1}\frac{1}{n^{2k}}
\]
for all even integers $2k\ge 2$. The value of $\zeta(2k)$ is given as a
rational multiple of $\pi^{2k}$. More precisely, Euler showed that \cite[Page 16]{Rademacher}
\begin{equation}\label{Eq2}
\zeta(2k)=\frac{(-1)^{k+1}B_{2k}(2\pi)^{2k}}{2(2k)!},
\end{equation}
where the Bernoulli numbers $B_m$ are defined by the exponential generating function \[\frac{x}{e^x-1}=\sum_{m=0}^\infty B_m\frac{x^m}{m!}.\] 

Since Euler's time, numerous proofs of his formula for $\zeta(2k)$ have been
given. Our goal here is to give yet another proof of \eqref{Eq2} that involves a
new identity for Bernoulli numbers and a different induction process than
used previously. We consider the Fourier coefficients $a_n(k)$ and $b_n(k)$
of the function that is periodic of period $2\pi$ and is given by $f(x)=x^k$ on the interval $(-\pi,\pi]$, and we obtain a 
pair of intertwining recurrences for these coefficients
(see \eqref{Eq7} and \eqref{Eq8} below). We then apply Parseval's theorem to connect these Fourier
coefficients with $\zeta (2k)$ and then establish (1) by solving this
recurrence. The novelty in our approach is a new identity for Bernoulli
numbers that we obtain and use.

A very recent paper by Navas, Ruiz, and Varona \cite{Navas} establishes new
connections between the Fourier coefficients of many fundamental
sequences of polynomials such as the Legendre polynomials and the Gegenbauer
polynomials by making them periodic of period $1$. In doing so, these
authors consider the Fourier coefficients of $x^k$, but they do not use
their ideas to establish Euler's formula (1). Another paper by Kuo \cite{Kuo}
provides a recurrence for the values of $\zeta(2k)$ involving only the
values $\zeta(2j)$ for $j\le k/2$ instead of requiring $j\le k-1$ as we do.
Our method is very different from Kuo's. 

\section{Warm Up}
For each positive integer $k$, let $a_n(k)$ and $b_n(k)$ be the $n^\text{th}$ Fourier coefficients of the function $x\mapsto x^k$. That is, $a_0(k)=\displaystyle{\frac{1}{2\pi}\int_{-\pi}^\pi x^k\,dx}$, $a_n(k)=\displaystyle{\frac{1}{\pi}\int_{-\pi}^\pi x^k\cos(nx)\,dx}$, and $b_n(k)=\displaystyle{\frac{1}{\pi}\int_{-\pi}^\pi x^k\sin(nx)\,dx}$ for $n,k\geq 1$. Parseval's identity \cite[page 191]{Rudin} informs us that 
\begin{equation}\label{Eq1}
\frac{1}{\pi}\int_{-\pi}^\pi x^{2k}\,dx=2a_0(k)^2+\sum_{n=1}^\infty(a_n(k)^2+b_n(k)^2).
\end{equation} 

Let us apply \eqref{Eq1} in the case $k=1$. We easily calculate that $a_0(1)=a_n(1)=0$ and $b_n(1)=2\dfrac{(-1)^{n+1}}{n}$ for all $n\geq 1$. This implies that \[\frac{2\pi^2}{3}=\frac{1}{\pi}\int_{-\pi}^\pi x^2\,dx=\sum_{n=1}^\infty b_n(1)^2=\sum_{n=1}^\infty\frac{4}{n^2},\] so we obtain Euler's classic result $\zeta(2)=\dfrac{\pi^2}{6}$. In Section 4, we use Parseval's identity to obtain a new inductive proof of Euler's famous formula \eqref{Eq2}.

\section{An Identity for Bernoulli Numbers}Before we proceed, let us recall some well-known properties of Bernoulli numbers (see \cite[Chapter 1]{Rademacher}). In Lemma \ref{Lem1}, we also establish one new identity involving these numbers. 

The first few Bernoulli numbers are $B_0=1$, $B_1=-\frac 12$, and $B_2=\frac 16$. If $m\geq 3$ is odd, then $B_m=0$. The equation \begin{equation}\label{Eq3}
\sum_{m=0}^{n-1}B_m{n\choose m}=0
\end{equation} holds for any integer $n\geq 2$. If $n$ is odd, then \begin{equation}\label{Eq4}
\sum_{m=0}^{n}B_m2^m{n\choose m}=0.
\end{equation}
One can prove \eqref{Eq4} by evaluating the Bernoulli polynomial $B_n(x)=\sum_{\ell=0}^n B_\ell{n\choose\ell}x^{n-\ell}$ at $x=1/2$ and using the known identity $B_n(x)=(-1)^nB_n(1-x)$. 

In what follows, we write ${n\choose r_1,r_2,r_3}$ to denote the trinomial coefficient given by $\dfrac{n!}{r_1!r_2!r_3!}$. 
\begin{lemma}\label{Lem1}
For any positive integer $k$, \[\sum_{i+t\leq\left\lfloor k/2\right\rfloor}B_{2t}2^{2t}\textstyle{2k+2\choose 2t,2i+1,2k-2t-2i+1}=(k+1)\left(2^{2k}+(-1)^k{2k\choose k}\right),\] where the sum ranges over all nonnegative integers $i$ and $t$ satisfying $i+t\leq \left\lfloor k/2\right\rfloor$.
\end{lemma} 
\begin{proof}
The difficulty in evaluating the given sum spawns from its unusual limits of summation. Therefore, we will first evaluate the sum obtained by allowing $i$ and $t$ to range over all nonnegative integers satisfying $i+t\leq k$. Because $\sum_{i=0}^{k-t}{2(k-t)+2\choose 2i+1}=2^{2k-2t+1}$, we have \[\sum_{i+t\leq k}B_{2t}2^{2t}{\textstyle{2k+2\choose 2t,2i+1,2k-2t-2i+1}}=\sum_{t=0}^k B_{2t}2^{2t}{\textstyle{2k+2\choose 2t}}\sum_{i=0}^{k-t}{\textstyle{2(k-t)+2\choose 2i+1}}=2^{2k+1}\sum_{t=0}^{k}B_{2t}{\textstyle{2k+2\choose 2t}}.\] Since $B_m=0$ for all odd $m\geq 3$, 
\begin{equation}\label{Eq5}
\sum_{i+t\leq k}B_{2t}2^{2t}{\textstyle{2k+2\choose 2t,2i+1,2k-2t-2i+1}}=2^{2k+1}\left(\sum_{\ell=0}^{2k+1}B_\ell{\textstyle{2k+2\choose \ell}}-B_1{\textstyle{2k+2\choose 1}}\right)=2^{2k+1}(k+1).
\end{equation} Note that we used \eqref{Eq3} along with the fact that $B_1=-\frac 12$ to deduce the last equality above. 

We next compute \[\sum_{\left\lfloor k/2\right\rfloor<i+t\leq k}B_{2t}2^{2t}{\textstyle{2k+2\choose 2t,2i+1,2k-2t-2i+1}}=\sum_{m=\left\lfloor k/2\right\rfloor+1}^k{\textstyle{2k+2\choose 2m+1}}\sum_{t=0}^m B_{2t}2^{2t}{\textstyle{2m+1\choose 2t}}\] \[=\sum_{m=\left\lfloor k/2\right\rfloor+1}^k{\textstyle{2k+2\choose 2m+1}}\left(\sum_{\ell=0}^{2m+1} B_\ell2^\ell{\textstyle{2m+1\choose \ell}}-2B_1{\textstyle{2m+1\choose 1}}\right)=\sum_{m=\left\lfloor k/2\right\rfloor+1}^k{\textstyle{2k+2\choose 2m+1}}(2m+1),\] where we have used \eqref{Eq4} to see that $\sum_{\ell=0}^{2m+1} B_\ell2^\ell{\textstyle{2m+1\choose \ell}}=0$. Therefore,  \[\sum_{\left\lfloor k/2\right\rfloor<i+t\leq k}B_{2t}2^{2t}{\textstyle{2k+2\choose 2t,2i+1,2k-2t-2i+1}}=\sum_{m=\left\lfloor k/2\right\rfloor+1}^k(2k+2)\left[{\textstyle{2k\choose 2m}}+{\textstyle{2k\choose 2m-1}}\right]\] 
\begin{equation}\label{Eq6}
=(k+1)\sum_{m=\left\lfloor k/2\right\rfloor+1}^k\left[{\textstyle{2k\choose 2m}}+{\textstyle{2k\choose 2k-2m}}+{\textstyle{2k\choose 2m-1}}+{\textstyle{2k\choose 2k-2m+1}}\right]=(k+1)\left(2^{2k}-(-1)^k{\textstyle{2k\choose k}}\right).
\end{equation} The proof of Lemma \ref{Lem1} now follows if we subtract \eqref{Eq6} from \eqref{Eq5}. 
\end{proof}
\section{The Formula for $\zeta(2k)$}
To begin this section, let us use integration by parts to see that for any $n\geq 1$ and $k\geq 2$, 
\begin{equation}\label{Eq7}
a_n(k)=\frac{1}{\pi}\int_{-\pi}^\pi x^k\cos(nx)\,dx=-\frac kn\frac{1}{\pi}\int_{-\pi}^\pi x^{k-1}\sin(nx)\,dx=-\frac kn\,b_n(k-1).
\end{equation}
Similarly, if $k$ is odd, then 
\begin{equation}\label{Eq8}
b_n(k)=\frac{1}{\pi}\int_{-\pi}^\pi x^k\sin(nx)\,dx=2\frac{(-1)^{n+1}\pi^{k-1}}{n}+\frac kn\,a_n(k-1).
\end{equation}
Now, it is clear from the definitions of $a_n(k)$ and $b_n(k)$ that $a_n(k)=0$ whenever $k$ is odd and $b_n(k)=0$ whenever $k$ is even. Hence, $a_n(k)^2+b_n(k)^2$ is actually equal to $(a_n(k)+b_n(k))^2$. If we appeal to \eqref{Eq7} and \eqref{Eq8} recursively and use the fact that $b_n(1)=2\dfrac{(-1)^{n+1}}{n}$, then a simple inductive argument shows that when $n,k\geq 1$, 
\begin{equation}\label{Eq9}
a_n(k)+b_n(k)=\sum_{\ell=0}^{\left\lfloor\frac{k-1}{2}\right\rfloor} c_n(k,\ell)\frac{\pi^{2\ell}}{n^{k-2\ell}},
\end{equation}
where we define
\begin{equation}\label{Eq10}
c_n(k,\ell)=\begin{cases} \frac{2k!}{(2\ell+1)!}(-1)^{\left\lfloor k/2\right\rfloor+\ell+n+1}, & \mbox{if } 0\leq\ell\leq \left\lfloor\frac{k-1}{2}\right\rfloor; \\ 0, & \mbox{otherwise. } \end{cases}
\end{equation}

Gathering \eqref{Eq1}, \eqref{Eq9}, and \eqref{Eq10} together yields \[2\frac{\pi^{2k}}{2k+1}-2a_0(k)^2=\frac{1}{\pi}\int_{-\pi}^\pi x^{2k}\,dx-2a_0(k)^2=\sum_{n=1}^\infty\left(\sum_{\ell=0}^{\left\lfloor\frac{k-1}{2}\right\rfloor}c_n(k,\ell)\frac{\pi^{2\ell}}{n^{k-2\ell}}\right)^2\] 
\begin{equation}\label{Eq12} 
=\sum_{n=1}^\infty\sum_{j=0}^{2\left\lfloor\frac{k-1}{2}\right\rfloor}r_n(k,j)\frac{\pi^{2j}}{n^{2k-2j}}=\sum_{j=0}^{2\left\lfloor\frac{k-1}{2}\right\rfloor}r_n(k,j)\pi^{2j}\zeta(2k-2j),
\end{equation} where \[r_n(k,j)=\sum_{i=0}^jc_n(k,i)c_n(k,j-i)=\sum_{i=\max\left(0,j-\left\lfloor\frac{k-1}{2}\right\rfloor\right)}^{\min\left(\left\lfloor\frac{k-1}{2}\right\rfloor,j\right)}c_n(k,i)c_n(k,j-i)\] \[=\sum_{i=\max\left(0,j-\left\lfloor\frac{k-1}{2}\right\rfloor\right)}^{\min\left(\left\lfloor\frac{k-1}{2}\right\rfloor,j\right)}\frac{2k!}{(2i+1)!}(-1)^{\left\lfloor\frac{k-1}{2}\right\rfloor+i+n+1}\frac{2k!}{(2j-2i+1)!}(-1)^{\left\lfloor\frac{k-1}{2}\right\rfloor+j-i+n+1}\] \begin{equation}\label{Eq13}
=\frac{4(-1)^jk!^2}{(2j+2)!}\sum_{i=\max\left(0,j-\left\lfloor\frac{k-1}{2}\right\rfloor\right)}^{\min\left(\left\lfloor\frac{k-1}{2}\right\rfloor,j\right)}{2j+2\choose 2i+1}.
\end{equation} 
Let $\mu_j(k)=\displaystyle{\sum_{i=\max\left(0,j-\left\lfloor\frac{k-1}{2}\right\rfloor\right)}^{\min\left(\left\lfloor\frac{k-1}{2}\right\rfloor,j\right)}{2j+2\choose 2i+1}}$. If $j\leq\left\lfloor\frac{k-1}{2}\right\rfloor$, then $\mu_j(k)=\displaystyle{\sum_{i=0}^j{2j+2\choose 2i+1}=2^{2j+1}}$. If $\left\lfloor\frac{k-1}{2}\right\rfloor<j\leq k-1$, then \[\mu_j(k)=2^{2j+1}-\sum_{i=0}^{j-\left\lfloor\frac{k-1}{2}\right\rfloor-1}{2j+2\choose 2i+1}-\sum_{i=\left\lfloor\frac{k-1}{2}\right\rfloor+1}^{j}{2j+2\choose 2i+1}=2^{2j+1}-2\sum_{i=0}^{j-\left\lfloor\frac{k-1}{2}\right\rfloor-1}{2j+2\choose 2i+1}.\] 

Assume inductively that we have proven \eqref{Eq2} when $k$ is replaced by any smaller positive integer. Using \eqref{Eq12}, \eqref{Eq13}, and this induction hypothesis, we find that \[2\frac{\pi^{2k}}{2k+1}-2a_0(k)^2=r_n(k,0)\zeta(2k)+\sum_{j=1}^{2\left\lfloor\frac{k-1}{2}\right\rfloor}\frac{4(-1)^jk!^2}{(2j+2)!}\pi^{2j}\zeta(2k-2j)\mu_j(k)\] \[=4k!^2\zeta(2k)+\sum_{j=1}^{2\left\lfloor\frac{k-1}{2}\right\rfloor}\frac{4(-1)^jk!^2}{(2j+2)!}\pi^{2j}\frac{(-1)^{k-j+1}B_{2k-2j}(2\pi)^{2k-2j}}{2(2k-2j)!}\mu_j(k)\] \[=4k!^2\zeta(2k)+2^{2k+2}(-1)^{k+1}\pi^{2k}k!^2\sum_{j=1}^{k-1}\frac{B_{2k-2j}\mu_j(k)}{2^{2j+1}(2j+2)!(2k-2j)!}\] (changing the limits of summation in the last line above is valid because $\mu_{k-1}(k)=0$ when $k$ is even).
Consequently, 
\[\frac{2(2k)!\zeta(2k)}{(2\pi)^{2k}}=\left(\frac{2^{-2k}}{2k+1}-\frac{a_0(k)^2}{(2\pi)^{2k}}\right){2k\choose k}-(-1)^{k+1}2(2k)!\sum_{j=1}^{k-1}\frac{B_{2k-2j}\mu_j(k)}{2^{2j+1}(2j+2)!(2k-2j)!}\] \[=\left(\frac{2^{-2k}}{2k+1}-\frac{a_0(k)^2}{(2\pi)^{2k}}\right){2k\choose k}+(-1)^k2(2k)!\sum_{j=1}^{k-1}\frac{B_{2k-2j}}{(2j+2)!(2k-2j)!}\]
\begin{equation}\label{Eq14}
-(-1)^k2(2k)!\sum_{j=\left\lfloor\frac{k-1}{2}\right\rfloor+1}^{k-1}\frac{B_{2k-2j}\sum_{i=0}^{j-\left\lfloor\frac{k-1}{2}\right\rfloor-1}{2j+2\choose 2i+1}}{2^{2j}(2j+2)!(2k-2j)!}.
\end{equation} 

Invoking the identity \eqref{Eq3}, we may write \[2(2k)!\sum_{j=1}^{k-1}\frac{B_{2k-2j}}{(2j+2)!(2k-2j)!}=\frac{2}{(2k+1)(2k+2)}\sum_{m=1}^{k-1}B_{2m}{\textstyle{2k+2\choose 2m}}\] 
\begin{equation}\label{Eq15}
=\frac{2}{(2k+1)(2k+2)}\left(k-B_{2k}{\textstyle{2k+2\choose 2k}}\right)=\frac{2k}{(2k+1)(2k+2)}-B_{2k}.
\end{equation} 
Furthermore, we may apply Lemma \ref{Lem1} to see that \[2(2k)!\sum_{j=\left\lfloor\frac{k-1}{2}\right\rfloor+1}^{k-1}\frac{B_{2k-2j}\sum_{i=0}^{j-\left\lfloor\frac{k-1}{2}\right\rfloor-1}{2j+2\choose 2i+1}}{2^{2j}(2j+2)!(2k-2j)!}=2\sum_{t=1}^{\left\lfloor k/2\right\rfloor}\frac{B_{2t}\sum_{i=0}^{\left\lfloor k/2\right\rfloor-t}{2k-2t+2\choose 2i+1}}{2^{2k-2t}(2k+1)(2k+2)}{\textstyle{2k+2\choose 2t}}\] \[=\frac{2^{1-2k}}{(2k+1)(2k+2)}\left[\sum_{i+t\leq\left\lfloor k/2\right\rfloor}B_{2t}2^{2t}{\textstyle{2k+2\choose 2t,2i+1,2k-2t-2i+1}}-\sum_{i=0}^{\left\lfloor k/2\right\rfloor}{\textstyle{2k+2\choose 2i+1}}\right]\] 
\begin{equation}\label{Eq16}
=\frac{2^{1-2k}}{(2k+1)(2k+2)}\left[(k+1)\left(2^{2k}+(-1)^k{\textstyle{2k\choose k}}\right)-\left(2^{2k}+{\textstyle{2k+1\choose k}}e_k\right)\right],
\end{equation}
where $e_k=0$ if $k$ is odd and $e_k=1$ if $k$ is even. 

If $k$ is odd, then \eqref{Eq14}, \eqref{Eq15}, and \eqref{Eq16} combine to show that $\displaystyle{\frac{2(2k)!\zeta(2k)}{(2\pi)^{2k}}}$ is equal to \[\frac{2^{1-2k}}{(2k+1)(2k+2)}\left[(k+1){\textstyle{2k\choose k}}-k2^{2k}+(k+1)\left(2^{2k}-{\textstyle{2k\choose k}}\right)-2^{2k}\right]+B_{2k},\] which simplifies to $B_{2k}$. This yields \eqref{Eq2} when $k$ is odd. If $k$ is even, $a_0(k)^2=\dfrac{\pi^{2k}}{(k+1)^2}$. Hence, we may again invoke \eqref{Eq14}, \eqref{Eq15}, and \eqref{Eq16} when $k$ is even to find that \[\frac{2(2k)!\zeta(2k)}{(2\pi)^{2k}}=\left(\frac{2^{-2k}}{2k+1}-\frac{a_0(k)^2}{(2\pi)^{2k}}\right){\textstyle{2k\choose k}}+\frac{2k}{(2k+1)(2k+2)}-B_{2k}\] \[-\frac{2^{1-2k}}{(2k+1)(2k+2)}\left[(k+1)\left(2^{2k}+{\textstyle{2k\choose k}}\right)-\left(2^{2k}+{\textstyle{2k+1\choose k}}\right)\right]\] \[=\frac{2^{1-2k}}{(2k+1)(2k+2)}\left[\frac{k^2}{k+1}{\textstyle{2k\choose k}}+k2^{2k}-(k+1)\left(2^{2k}+{\textstyle{2k\choose k}}\right)+\left(2^{2k}+{\textstyle{2k+1\choose k}}\right)\right]-B_{2k}\] \[=\frac{2^{1-2k}}{(2k+1)(2k+2)}\left[\frac{k^2}{k+1}{\textstyle{2k\choose k}}-(k+1){\textstyle{2k\choose k}}+\frac{2k+1}{k+1}{\textstyle{2k\choose k}}\right]-B_{2k}=(-1)^{k+1}B_{2k}.\] This proves \eqref{Eq2} when $k$ is even, completing the induction.

\end{document}